\documentclass[12pt]{article}
\usepackage{ct}

\usepackage{physics}
\newcommand\brho{\operatorname{\boldsymbol{\rho}}}
\newcommand{\Mod}[1]{\ (\mathrm{mod}\ #1)}
\newcommand\rounddown[1]{\left\lfloor#1\right\rfloor}
\newcommand\mult{\operatorname{\textup{{\fontfamily{ptm}\selectfont mult}}}}
\newcommand\dg{\operatorname{\textup{{\fontfamily{ptm}\selectfont deg}}}}
\providecommand{\dotdiv}{
  \mathbin{
    \vphantom{+}
    \text{
      \mathsurround=0pt 
      \ooalign{
        \noalign{\kern-.35ex}
        \hidewidth$\smash{\cdot}$\hidewidth\cr 
        \noalign{\kern.35ex}
        $-$\cr 
      }%
    }%
  }%
}

\specs{n}{vol (issue)}{year}{}

\dateline{Dec 20, 2021}{July 1, 2025}{TBD}

\keywords{Latin Square, Embedding, $(g,f)$-factors,  Cruse's Theorem, Andersen-Hoffman's Theorem, Ryser's Theorem, Amalgamation, Detachment}

\MSC{05B15, 05C70, 05C15}

\title{Ryser's Theorem for Symmetric $\brho$-latin Squares}

\author[1]{Amin Bahmanian}
\author[2]{A. J. W. Hilton}

\affil[1]{%
Department of Mathematics, Illinois State University, Normal, IL, U.S.A.

\email{mbahman@ilstu.edu}%
}

\affil[2]{%
Department of Mathematics, University of Reading, Reading, U.K. 

Department of Mathematics, Queen Mary University of London, London, U.K.

\email{a.j.w.hilton@reading.ac.uk}
}

\begin{document}

\maketitle


\begin{abstract}
  Let $L$ be an $n\times n$ array whose  top left $r\times r$ subarray is filled with $k$ different symbols, each occurring at most once in each row and at most once in each column. We establish necessary and sufficient conditions that ensure  the remaining cells of $L$ can be filled  such  that each symbol occurs at most once in each row and at most once in each column, $L$ is symmetric with respect to the main diagonal, and 
each symbol occurs a prescribed number of times in $L$. The case where the prescribed number of times each symbol occurs is $n$ was solved by Cruse  (J. Combin. Theory Ser. A 16 (1974), 18--22), and the case where the top left subarray is $r\times n$ and the symmetry is not required, was settled by 
Goldwasser et al. (J. Combin. Theory Ser. A 130 (2015), 26--41). Our result allows  the entries of the main diagonal to be specified as well, which leads to an extension of the Andersen-Hoffman Theorem (Annals of Disc. Math. 15 (1982) 9--26, European J. Combin. 4 (1983) 33--35). 
\end{abstract}


\section{Introduction} 
Throughout this paper, $n$ and $k$ are positive integers,  $[k]:=\{1,2,\dots, k\}$, and $\vb*{\rho} :=(\rho_1,\dots,\rho_k)$ with $1\leq \rho_\ell\leq n\leq k$ for $\ell \in [k]$ such that $\sum_{\ell \in [k]} \rho_\ell=n^2$. A {\it $\vb*{\rho} $-latin square} $L$ of {\it order} $n$ is an $n\times n$ array filled with $k$ different {\it symbols}, $[k]$, each occurring at most once in each row and at most once in each column, such that each symbol $\ell$ occurs exactly $\rho_\ell$ times in $L$ for $\ell \in [k]$, and $L$ is {\it symmetric}  if $L_{ij}=L_{ji}$ for  $i,j\in [n]$. 
An $r \times s$ {\it  $\vb*{\rho} $-latin rectangle} on the set $[k]$ of symbols is an $r\times s$ array in which each symbol in $[k]$ occurs at most once in each row and in each column, and in which   each symbol $i$ occurs at most $\rho_\ell$ times  for $\ell \in [k]$. A {\it latin square} (or {\it rectangle}) is a $\vb*{\rho} $-latin square (or {\it rectangle}) with ${\vb*{\rho}} =(n,\dots,n)$.

We are interested in the following problem.
\begin{question} \label{rhosymrysprob}
Let $n> \min \{r,s\}$. Find necessary and sufficient conditions that ensure that a symmetric $r \times s$ $\vb*{\rho} $-latin rectangle $L$ can be extended to a symmetric  $n\times n$ $\vb*{\rho}$-latin square $L'$. 
\end{question}
We remark that even the case $\vb*{\rho} =(n,\dots,n)$ of Problem \ref{rhosymrysprob} is  far from being settled. The following result of Cruse  which can be viewed as an analogue of Ryser's theorem \cite{MR42361} for partial symmetric latin squares,  resolves the case $r=s, \vb*{\rho} =(n,\dots,n)$ of  Problem \ref{rhosymrysprob}.  
For $\ell \in [k]$, let $e_\ell$ be the number of occurrences of  $\ell$ in $L$. 
\begin{theorem} \cite[Theorem 1]{MR329925} \label{crusorigthm}
An $r\times r$ symmetric latin rectangle $L$ on  $[n]$  can be extended to an $n \times n$ symmetric latin square if and only if  
\begin{enumerate} 
\item [\textup{(i)}]  $e_\ell\geq 2r-n$ for  $\ell \in [n]$;
\item [\textup{(ii)}] $\big|\{\ell \in [n] \ |\ e_\ell \equiv n \Mod 2\}\big|\geq r$.
\end{enumerate}
\end{theorem}
Let $\vb*{d}:=(d_1,\dots, d_k)$ be the {\it diagonal tail} of $L'$ where $d_\ell$ is the minimum number of occurrences of $\ell$ in $\{L'_{ii}\ |\ r+1\leq i \leq n\}$ and $\sum_{\ell\in [k]} d_\ell\leq n-r$  (we refer to the diagonal tail as the {\it diagonal} if $r=s=0$). Observe that if $L$ is extended to $L'$, then by permuting rows and columns of $L'$, one can obtain another $\vb*{\rho}$-latin square $L''$ containing $L$ with the same diagonal tail as $L'$ such that the entries of  $L''$ in positions $(i,i)$ for $i>r$ are in a prescribed order. Hence, in order to extend $L$ to $L'$ whose diagonal entries are specified, it is enough to extend  $L$ to $L'$ with a specified diagonal tail. Andersen,  and independently, Hoffman,  proved the following extension of  Cruse's Theorem. 
\begin{theorem} \cite{MR772580, MR694466} \label{rbyrsymprediag}
An $r\times r$ symmetric latin rectangle $L$ on $[n]$ can be extended to an $n \times n$ symmetric latin square with a prescribed diagonal tail $\vb*{d}$ with $\sum_{\ell\in [n]} d_\ell=n-r$ if and only if 
\begin{enumerate}
\item [\textup{(i)}] $e_\ell\geq 2r-n+d_\ell$ for  $\ell \in [n]$;
\item [\textup{(ii)}] $e_\ell+ d_\ell\equiv n \Mod 2$ for  $\ell \in [n]$. 
\end{enumerate}
\end{theorem}
Bryant and Rodger \cite{MR2029313} solved the   case $r\in \{1,2\},s=n, \vb*{\rho} =(n,\dots, n)$ of  Problem \ref{rhosymrysprob}.  Goldwasser et al.  \cite{MR3280683} found necessary and sufficient conditions under which an $r\times n$  $\vb*{\rho} $-latin rectangle can be extended to an $n\times n$ $\vb*{\rho} $-latin square, generalizing a classical result of Hall \cite{MR13111}. Most recently,  Bahmanian \cite{MR4414830} extended this further by establishing  necessary and sufficient conditions that ensure an $r\times s$  $\vb*{\rho} $-latin rectangle can be extended to an $n\times n$ $\vb*{\rho} $-latin square, generalizing Ryser's theorem \cite{MR42361}.
For a survey of results on  embedding of latin squares and related structures, cycle systems and graph designs, we refer the reader to  \cite{MR1275861}.

Let $L$ be an $r\times r$ symmetric $\vb*{\rho} $-latin rectangle. Let $i\in [r]$ be a row of $L$, $\ell\in [k]$ be a symbol, $I\subseteq [r]$ be a subset of rows, and  $K\subseteq [k]$ be a subset of symbols. Then $\mu_K(i)$  and $\mu_I(\ell)$ denote  the number of symbols in $K$ that are missing in row $i$, and the number of rows in $I$ where symbol $\ell$ is missing, respectively. Notice that 
\begin{align*}
    \mu_{[r]} (\ell)&=r-e_\ell, &0\leq \mu_I(\ell)&\leq \min\{|I|, r-e_\ell\} &\forall \ell\in [k], I\subseteq [r],\\
    \mu_{[k]} (i)&=k-r, &0\leq \mu_K(i)&\leq \min\{|K|, k-r\}&\forall i\in [r], K\subseteq [k].
\end{align*}

The complement of a set $S$ is denoted by $\bar S$, and $x \dotdiv y:=\max\{0, x-y\}$. If $x,y,z$ are non-negative, then $(x\dotdiv y)\dotdiv z=(x-y) \dotdiv z$. Whenever it is not ambiguous, we write  $x-y \dotdiv z$ instead of  $(x-y) \dotdiv z$. 

Suppose that an  $r\times r$ symmetric $\vb*{\rho} $-latin rectangle $L$ is extended to an  $n\times n$ symmetric $\vb*{\rho} $-latin square $L'$ with a prescribed diagonal tail $\vb * d$, and let us fix a symbol $\ell \in [k]$. Without loss of generality we assume that $L'_{ii}=\ell$ for $r+1\leq i\leq r+d_\ell$. On the one hand, there are $\rho_\ell-e_\ell-d_\ell$ occurrences of  $\ell$ outside the top left $(r+d_\ell)\times (r+d_\ell)$ subsquare. On the other hand, there are at most $n-r-d_\ell$ occurrences of  $\ell$ in the last $n-r-d_\ell$ rows, and at most $n-r-d_\ell$ occurrences of  $\ell$ in the last $n-r-d_\ell$ columns. Therefore, we must have
\begin{align} \label{easyneccondiag}
\rho_\ell-e_\ell +d_\ell\leq 2(n-r)\quad \forall \ell \in [k].
\end{align}
Due to the symmetry of $L$ and $L'$, off-diagonal entries occur in pairs, and so, if for some $\ell \in [k]$, $\rho_\ell-e_\ell-d_\ell$ is odd,  then $\ell$ must occur at least one more time on the diagonal, and  consequently, 
\begin{align} \label{congcon1diag} 
\big|\{ \ell \in [k] \ | \ \rho_\ell-e_\ell\not\equiv d_\ell \Mod 2  \}\big|\leq n-r-\sum_{\ell\in [k]} d_\ell.
\end{align}
Moreover,
\begin{align} \label{congcon2}
   n-r - \sum_{\ell\in [k]} d_\ell\equiv \sum_{\ell\in [k]} (\rho_\ell-e_\ell-d_\ell)\equiv  \big|\{ \ell \in [k] \ | \ \rho_\ell-e_\ell\not\equiv d_\ell \Mod 2  \}\big| \Mod 2\quad \forall \ell \in [k].
\end{align}
An $r\times r$ symmetric $\vb*{\rho} $-latin rectangle $L$  is {\it $(\vb*{\rho}, \vb*{d})$-admissible} if it meets conditions \eqref{easyneccondiag}--\eqref{congcon2}.  Here is our main result. 
\begin{theorem} \label{andhoffrhothm}
An $r\times r$ symmetric $\vb*{\rho} $-latin rectangle $L$  can be completed to  an $n \times n$ symmetric $\vb*{\rho} $-latin square  with a prescribed diagonal tail $\vb*{d}$ if and only if $L$ is $(\vb*{\rho}, \vb*{d})$-admissible and 
\begin{align}
(n-r)(r-|I|)+\sum_{\ell\in \bar K}\rounddown{\frac{\rho_\ell-e_\ell-d_\ell}{2}} &\geq  \sum_{\ell\in K} \Big( \rho_\ell-e_\ell-n+r \dotdiv \mu_I(\ell) \Big) \nonumber\\ 
  &\quad + \sum_{i\in I} \Big( n-r \dotdiv \mu_{K}(i) \Big) \quad \forall I\subseteq [r], K\subseteq [k]\label{reallylongineqdial}.
\end{align}
\end{theorem}
\begin{remark}
As we shall see in Section \ref{AndHoffhoSec}, Condition \eqref{reallylongineqdial} can be replaced by the following  two conditions. 
\begin{align}
(n-r)|I|&\leq \sum_{\ell\in [k]} \min \Big \{ \rounddown{\frac{\rho_\ell-e_\ell-d_\ell}{2}} ,\mu_I(\ell)\Big\}
 &\forall I\subseteq [r],\label{longineqdial1} \\
\sum_{\ell\in K}(\rho_\ell-e_\ell+r\dotdiv n)&\leq \sum_{i\in [r]} \min \Big \{ n-r, \mu_K(i)\bigg \}
 &\forall K\subseteq [k].\label{longineqdial2} 
\end{align}
\end{remark}
At first sight, it may not be obvious why Theorem \ref{andhoffrhothm} implies Theorems \ref{crusorigthm} and \ref{rbyrsymprediag}; see Remarks \ref{remarkdiago1}. Both Corollaries \ref{cor2diag} and  \ref{hellcordiag} offer  simpler generalizations of Theorems \ref{crusorigthm} and \ref{rbyrsymprediag}. 

Using detachments we reduce the completion of the desired partial $\rho$-latin  square $A$ to finding a subgraph with prescribed degree sequences of an auxiliary  graph associated with $A$. We complete the proof by applying the celebrated Lov\'{a}sz's  $(g,f)$-factor Theorem. Intuitively speaking, we use matching theory to decide which $n-r$ symbols among $k-r$ available symbols of each row should be chosen, and then we use detachment theory to arrange the chosen symbols in such a way that the latin property is maintained. 

Further terminology along with the two main tools  are discussed in Section \ref{termsec}.  Theorems  \ref{andhoffrhothm} is proven in Section \ref{AndHoffhoSec}. We conclude the paper with a few corollaries and  open problems.

\section{Prerequisites} \label{termsec} 
For a graph $G=(V,E)$,   $u\in V, e\in E$, and $K\subseteq V$, $\dg_G(u), \mult_G(e), \mult_G(uS)$ denote the number of edges incident with $u$,  the multiplicity of the edge $e$, and the number of edges between $u$ and $S$,  respectively. If the edges of $G$ are colored with $k$ colors (the set of colors is always $[k]$), then $G(\ell)$ is the color class $\ell$ of $G$ for $\ell \in [k]$. A bigraph $G$ with bipartition $\{X, Y \}$ will be denoted by $G[X, Y ]$,  
and if $S\subseteq X$, then $\bar{S}$ means $X\backslash S$.

For $i=1,2$, an edge that is incident with one vertex only is called an  {\it $i$-loop}
  if  it contributes $i$ to the degree of that vertex. 
If $e$ is a 1-loop incident with vertex $u$, we write $e=u$, and if $e$ is a 2-loop incident with vertex $u$, we write $e=u^2$.  We denote the set of 1-loops of a graph $G$ by $E^{1}(G)$, and $E^{2}(G):=E(G)\backslash E^{1}(G)$. 

Let $\mathbb{K}_n$ denote the $n$-vertex graph in which every pair of distinct vertices are adjacent and each vertex is incident with a $1$-loop. Observe  that $\mult_{\mathbb{K}_n}(u)=\mult_{\mathbb{K}_n}(uv)=1$ for $u,v \in V(\mathbb{K}_n)$ with $u\neq v$. There is a one-to-one correspondence between  symmetric latin squares of order $n$ and 1-factorizations of $\mathbb{K}_n$.

Let $G$ be a graph whose edges are colored, and let $\alpha\in V(G)$. By {\it splitting} $\alpha$ into $\alpha_1,\dots,\alpha_p$, we obtain a new graph $F$ whose vertex set is $\left(V(G)\backslash \{\alpha\}\right) \cup \{\alpha_1,\dots, \alpha_p\}$ so that each edge $\alpha u$ in $G$ becomes $\alpha_i u$ for some $i\in [p]$ in $F$, and each 1-loop $\alpha$ in $G$ becomes $\alpha_i$ for some $i\in [p]$ in $F$. Intuitively speaking, when we {\it split}  a vertex $\alpha$ into $\alpha_1,\dots,\alpha_p$, we share the edges incident with $\alpha$ among $\alpha_1\dots,\alpha_p$.   In this manner, $F$ is a {\it detachment} of $G$, and $G$ is an {\it amalgamation} of $F$ obtained by {\it identifying} $\alpha_1,\dots, \alpha_p$ by $\alpha$. We need the following detachment lemma. Here, $x\approx y$ means $\lfloor y \rfloor \leq x\leq \lceil y \rceil$. 
\begin{lemma}
\label{amalgambahcpc} \cite[Theorem 4.1]{MR2942724}
  Let $G$ be a graph whose edges are colored with $k$  colors, and let $\alpha\in V(G)$. There exists a graph $F$ obtained by splitting $\alpha$ into $\alpha_1,\dots,\alpha_p$ such that
 \begin{enumerate}
     \item [\textup{(i)}] $\dg_{F(\ell)}(\alpha_i)\approx \dg_{G(\ell)}(\alpha)/p$  for  $i\in [p],\ell\in [k]$;
     \item [\textup{(ii)}] $\mult_F(\alpha_i)\approx \mult_G(\alpha)/p$ for  $i\in[p]$;
     \item [\textup{(iii)}] $\mult_F(\alpha_i u)\approx \mult_G(\alpha u)/p$ for  $i\in[p],u\in V(G)\backslash \{\alpha\}$;
     \item  [\textup{(iv)}]  $\mult_F(\alpha_i \alpha_j)\approx \mult_G(\alpha^2)/\binom{p}{2}$ for  $i,j\in[p],i\neq j$.
 \end{enumerate} 
\end{lemma}
  Using this detachment lemma, it is easy to construct symmetric $\vb*{\rho} $-latin squares with prescribed diagonals.
\begin{theorem} 
For every  $n,k, \vb*{d} =(d_1,\dots,d_k), \vb*{\rho} =(\rho_1,\dots,\rho_k)$ with $1\leq \rho_1,\dots,\rho_k\leq n\leq k$, $\sum_{\ell \in [k]} \rho_\ell=n^2$ and $\sum_{\ell \in [k]} d_\ell\leq n$,  there exists a symmetric $\vb*{\rho} $-latin square of order $n$ with diagonal  $\vb*{d}$ if and only if 
\begin{align*}
&    \big|\{ \ell \in [k] \ | \ \rho_\ell\not\equiv d_\ell \Mod 2   \}\big| \leq n-\sum_{\ell \in [k]} d_\ell, \\ 
&    \big|\{ \ell \in [k] \ | \ \rho_\ell\not\equiv d_\ell \Mod 2  \}\big| \equiv n-\sum_{\ell \in [k]} d_\ell \Mod 2.
\end{align*}
\end{theorem}
\begin{proof}
The proof of necessity is very similar to those of \eqref{congcon1diag} and \eqref{congcon2}.  To prove the sufficiency, let  $G$ be a  graph with $V(G)=\{\alpha\}$, and $\mult_G(\alpha)=n, \mult_G(\alpha^2)=\binom{n}{2}$. First, we color the 1-loops of $G$ such that $\mult_{G(\ell)}(\alpha)=d_\ell$ for $\ell\in [k]$. Then we color an additional 1-loop with $\ell$ if $\rho_\ell-d_\ell$ is odd for $\ell\in [k]$. 
Since $|\{ \ell \in [k] \ | \ \rho_\ell\not\equiv d_\ell \Mod 2   \}| \leq n-\sum_{\ell \in [k]} d_\ell$, this is possible. The number of uncolored 1-loops is  $n-\sum_{\ell \in [k]} d_\ell-|\{ \ell \in [k] \ | \ \rho_\ell\not\equiv d_\ell \Mod 2  \}|$ which is even.  Therefore, we can color the uncolored 1-loops by coloring an even number of 1-loops with each color. Then, we color the remaining edges of $G$ such that  $\mult_{G(\ell)}(\alpha^2)=\big(\rho_\ell-\mult_{G(\ell)}(\alpha)\big)/2$ for $\ell \in [k]$. This is possible, because the coloring of 1-loops ensures $\rho_\ell-\mult_{G(\ell)}(\alpha)$ is even for $\ell \in [k]$, and  
$$
\sum_{\ell\in [k]} \big(\rho_\ell-\mult_{G(\ell)}(\alpha)\big)/2=(n^2-n)/2=\mult_G(\alpha^2).
$$
Applying the detachment lemma with $p=n$ yields the  graph $F\cong \mathbb{K}_n$ whose colored edges correspond to symbols in the desired $\vb*{\rho} $-latin square  $L$. More precisely,  $L$ is obtained by placing symbol $\ell$ in $L_{i j}$ and $L_{ ji}$ whenever the edge $\alpha_i \alpha_j$ is colored $\ell$ for $i\neq j$, and  placing symbol $\ell$ in $L_{i i}$ whenever the 1-loop $\alpha_i$ is colored $\ell$.
\end{proof}

 Let $f,g$ be  integer functions on the vertex set of a graph $G$ such that $0\leq g(x)\leq f(x)$ for all $x$. A {\it $(g,f)$-factor} is a spanning subgraph $F$ of $G$ with the property that $g(x)\leq \dg_F(x)\leq f(x)$ for each $x$.  
Let $G[X,Y]$ be a  bigraph.  By \cite[Theorem 5]{MR3564794} and \cite[Theorem 1]{MR1081839},  $G$ has a  $(g,f)$-factor   if and only if either one of the following two conditions hold.
\begin{align*}
    \sum_{a\in A} g(a)&\leq \sum_{a\in N_G(A)} \min \Big\{ f(a), \mult_G(aA)\Big\}
\quad &\forall A\subseteq X, A\subseteq Y,\\
\sum_{a\in A}f(a)&\geq  \sum_{a\notin A} \Big( g(a) \dotdiv \dg_{G-A}(a) \Big) \quad &\forall A\subseteq X\cup Y.
\end{align*}
Here, $N_G(A)$ is the  neighborhood of $A$ in $G$. We remark that both of these  results are special cases of the  Lov\'{a}sz's  $(g,f)$-factor Theorem \cite{MR325464}.

\section{Proof of Theorem \ref{andhoffrhothm}} \label{AndHoffhoSec}
We established the necessity of \eqref{easyneccondiag}--\eqref{congcon2} in the introduction. The necessity of the remaining conditions will be evident at the end of the proof. To prove the sufficiency, suppose that  $L$ is a $(\vb*{\rho}, \vb*{d})$-admissible $r\times r$ symmetric $\vb*{\rho} $-latin rectangle.  
Let $F=\mathbb K_{n}$ with $V(F)=\tilde X:=\{x_1,\dots,x_n\}$, and let $X=\{x_1,\dots, x_r\}$. In $F$, the 1-loop $x_i$ is colored $\ell$ if $L_{ii}=\ell$ for   $i\in[r]$, and  the edge $x_i x_{j}$ is colored $\ell$ if $L_{ij}=\ell$ for distinct  $i,j\in[r]$. Since $L$ is symmetric, this coloring is well-defined. Observe that some edges of $F$ are uncolored. For $\ell \in [k]$, we color $d_\ell$ arbitrary uncolored 1-loops (incident with the vertices of $\tilde X \backslash X$) with color $\ell$. We have $\dg_{F(\ell)}(u)\leq 1$ for $u\in \tilde X, \ell \in [k]$.   
Let  $G$ be the graph obtained by amalgamating $x_{r+1},\dots, x_n$ of $F$ into a single vertex $\alpha$, so $\mult_G(\alpha)=\mult_G(\alpha x_i)=n-r$ for $ i\in [r]$,  $\mult_G(\alpha^2)=\binom{n-r}{2}$, and  $\mult_{G(\ell)}(\alpha)=d_\ell$ for $\ell \in [k]$.

Let  $\Gamma [X, [k]]$ be the simple bigraph whose edge set is $$\{u \ell\ |\ u\in X, \ell \in [k],  \dg_{F(\ell)}(u)=0\}.$$ 
For  $i\in [r]$,  $\sum_{\ell\in [k]}\dg_{F(\ell)}(x_i)=r$, and so we have
\begin{align}\label{degeqGamma}
 \begin{cases}
\dg_\Gamma(x_i)=k-r & \mbox{if  } i \in [r], \\
\dg_{\Gamma}(\ell)=r-e_\ell& \mbox{if  } \ell \in [k].
\end{cases}
\end{align}

Observe that $L$  can be completed to  an $n \times n$ symmetric $\vb*{\rho} $-latin square   if and only if the uncolored edges of  $F$ can be colored so that
\begin{align}\label{colorcon1cruse}
\forall \ell \in [k], \quad \begin{cases}
\dg_{F(\ell)}(u)\leq 1 & \mbox{if  } u\in \tilde X, \\
|E^1(F(\ell))|+2|E^2(F(\ell))|=\rho_\ell.
\end{cases}
\end{align}

We show that the coloring of  $F$ can be completed such that  \eqref{colorcon1cruse} holds if and only if the coloring of $G$ can be completed such that
\begin{align}\label{colorcon2cruse}
\forall \ell \in [k]\quad \quad  \begin{cases}
\dg_{G(\ell)}(u) \leq 1 & \mbox{if  } u \in X, \\
\dg_{G(\ell)}(\alpha) \leq n-r, \\
|E^1(G(\ell))|+2|E^2(G(\ell))|= \rho_\ell.
\end{cases}
\end{align}
To see this, first assume that the coloring of  $F$ can be completed so that  \eqref{colorcon1cruse} holds. Identifying all the  vertices in $\tilde X\backslash X$   by $\alpha$, we will get the graph $G$ satisfying \eqref{colorcon2cruse}. Conversely, suppose that  we have a  coloring of $G$
such that \eqref{colorcon2cruse} holds. Applying the detachment lemma to $G$, we get a graph $F'$ obtained by splitting $\alpha$ into $\alpha_1,\dots,\alpha_{n-r}$, such that 
 \begin{enumerate}
     \item [\textup{(i)}] $\dg_{F'(\ell)}(\alpha_i)\approx\dg_{G(\ell)}(\alpha)/(n-r)\leq 1$  for  $i\in [n-r],\ell\in [k]$;
     \item [\textup{(ii)}] $\mult_{F'}(\alpha_i)= \mult_G(\alpha)/(n-r)=1$ for  $i\in[n-r]$;
     \item [\textup{(iii)}] $\mult_{F'}(\alpha_i u)= \mult_G(\alpha u)/(n-r)=1$ for  $i\in[n-r],u\in X$;
     \item [\textup{(iv)}]  $\mult_{F'}(\alpha_i \alpha_j)= \mult_G(\alpha^2)/\binom{n-s}{2}=1$ for distinct $i,j\in[n-r]$. 
\end{enumerate}
Since $F'\cong F$ and the coloring of $F'$ satisfies \eqref{colorcon1cruse}, we are done.

Since $L$ is $(\vb*{\rho}, \vb*{d})$-admissible, for $\ell \in [k]$ we have $2(n-r) \geq \rho_\ell - e_\ell+ d_\ell$, and so $\rho_\ell-e_\ell-n+r\leq (\rho_\ell-e_\ell-d_\ell)/2$.  Moreover, $\rho_\ell-e_\ell-n+r\leq r-e_\ell$ for $\ell \in [k]$. We show that the coloring of  $G$ can be completed such that  \eqref{colorcon2cruse} is satisfied if and only if there exists  a subgraph $\Theta$ of $\Gamma$  with $r(n-r)$ edges so that 
\begin{align}\label{colorcon3nodiag}
\begin{cases}
\dg_\Theta (x_i)=n-r & \mbox{if  }  i\in [r], \\
\rho_\ell-e_\ell-n+r\leq \dg_\Theta (\ell)\leq \dfrac{\rho_\ell-e_\ell-d_\ell}{2}  & \mbox{if  } \ell \in [k].
\end{cases}
\end{align}
 \noindent To prove this, suppose that the coloring $G$ can be completed such that  \eqref{colorcon2cruse} is satisfied. Let  $\Theta [X, [k]]$ be the bigraph whose edge set is 
$$\{u \ell\ |\ u\in X, \ell \in [k],\alpha u\in E(G(\ell))  \}.$$ 
It is clear that  $\Theta \subseteq \Gamma$.
For $i\in [r]$, 
 $\dg_\Theta(x_i)=\mult_G(\alpha x_i)=n-r$, and for $\ell \in [k]$,
  \begin{align*}
\rho_\ell&=|E^1(G(\ell))|+2|E^2(G(\ell))|\\&=e_\ell+\mult_{G(\ell)}(\alpha)+2\mult_{G(\ell)}(\alpha, X)+2\mult_{G(\ell)}(\alpha^2)\\
&\geq e_\ell+d_\ell+2\dg_\Theta (\ell).
\end{align*}
Thus,  $\dg_\Theta (\ell)\leq  (\rho_\ell-e_\ell-d_\ell)/2$ for $\ell \in [k]$. Moreover, 
\begin{align*}
    n-r&\geq \dg_{G(\ell)}(\alpha)=\dg_\Theta (\ell)+\mult_{G(\ell)}(\alpha)+2\mult_{G(\ell)}(\alpha^2)\\
        &=\dg_\Theta (\ell)+\mult_{G(\ell)}(\alpha)+\big(\rho_\ell-e_\ell-\mult_{G(\ell)}(\alpha)-2\dg_\Theta (\ell)\big)\\
        &=\rho_\ell-e_\ell-\dg_\Theta (\ell),
\end{align*} 
and so $\dg_\Theta (\ell)\geq \rho_\ell-e_\ell-n+r$. Conversely, suppose that   $\Theta\subseteq \Gamma$ satisfying \eqref{colorcon3nodiag} exists. For each $\ell\in [k]$, if $\ell x_i \in E(\Theta)$ for some $i\in [r]$, we color an $\alpha x_i$-edge in $G$ with $\ell$. Since $\dg_\Theta(x_i)=n-r$ for $i\in [r]$, all the edges between $\alpha$ and $X$ can be colored this way. Since $\Theta$ is simple, $d_{G(\ell)}(u)\leq 1$ for $\ell \in [k]$ and $u
\in X$. Let $O$ be the set of colors $\ell \in [k]$ such that $\rho_\ell-e_\ell-d_\ell$ is odd.  By \eqref{congcon1diag} and \eqref{congcon2}, $(n-r-|O|-\sum_{\ell\in [k]} d_\ell)/2$ is a non-negative integer. By \eqref{colorcon3nodiag}, $(\rho_\ell-e_\ell-d_\ell)/2-\dg_\Theta (\ell)\geq 0$ for $\ell \in [k]\backslash O$ and 
$(\rho_\ell-e_\ell-d_\ell-1)/2-\dg_\Theta (\ell)\geq 0$ for $i\in O$. Since 
\begin{align*}
    \dfrac{n-r-|O|-\sum_{\ell\in [k]} d_\ell}{2}&\leq \dfrac{(n-r)^2-|O|-\sum_{\ell\in [k]} d_\ell}{2}=\frac{n^2-r^2-(n-r)}{2}-r(n-r)-\frac{|O|}{2}\\
&\leq \sum_{\ell \in [k]}\dfrac{\rho_\ell-e_\ell-d_\ell}{2}-\sum_{\ell \in [k]} \dg_\Theta (\ell) - \frac{|O|}{2}\\
&=\sum_{\ell \in [k]\backslash O }\left (\dfrac{\rho_\ell-e_\ell-d_\ell}{2}-\dg_\Theta (\ell)\right) + \sum_{i\in  O}\left (\dfrac{\rho_\ell-e_\ell-d_\ell-1}{2}-\dg_\Theta (\ell)\right),
\end{align*}
there exists a sequence of integers $a_1,\dots, a_k$ such that 
\begin{align*}
\begin{cases}
a_1+\dots+a_k=\dfrac{n-r-|O|-\sum_{\ell\in [k]} d_\ell}{2}, \\
0\leq a_\ell\leq \dfrac{\rho_\ell-e_\ell-d_\ell}{2}-\dg_\Theta (\ell) \quad &\text{ for } \ell \in [k]\backslash O,\\
0\leq a_\ell\leq \dfrac{\rho_\ell-e_\ell-d_\ell-1}{2}-\dg_\Theta (\ell) \quad &\text{ for }  \ell\in O.
\end{cases}
\end{align*}
Now let 
\begin{align*}
\overline{d_\ell}=\begin{cases}
2a_\ell\quad &\text{ for } \ell \in [k]\backslash O,\\
2a_\ell+1 \quad &\text{ for }  \ell\in O.
\end{cases}
\end{align*}
The non-negative sequence $\overline{d_1},\dots,\overline{d_k}$ satisfies the following. 
\begin{align} \label{d_iseqcond}
\begin{cases}
\overline{d_1}+\dots+\overline{d_k}=n-r-\sum_{\ell\in [k]} d_\ell, \\
\overline{d_\ell} \equiv\rho_\ell-e_\ell-d_\ell  \Mod 2 & \text{ for } \ell \in [k],\\
\dg_\Theta (\ell)\leq \dfrac{\rho_\ell-e_\ell-d_\ell-\overline{d_\ell}}{2}& \text{ for } \ell \in [k].
\end{cases}
\end{align}
We color the uncolored  loops of $G$ such that there are $\overline{d_\ell}$ 1-loops colored $\ell$ for $\ell\in [k]$, (so $\mult_{G(\ell)}(\alpha)=d_\ell+\overline{d_\ell}$ for $\ell \in [k]$) and 
$$\mult_{G(\ell)}(\alpha^2) =\frac{1}{2}\left(\rho_\ell-e_\ell-d_\ell-\overline{d_\ell}\right)-\dg_\Theta (\ell)\quad  \forall \ell \in [k].$$
This is possible for \eqref{d_iseqcond} and
  \begin{align*}
\sum_{\ell \in [k]}\left(\rho_\ell-e_\ell-2\dg_\Theta (\ell)-d_\ell-\overline{d_\ell})\right)&=n^2-r^2-2r(n-r)-(n-r)\\
&=2\binom{n-r}{2}=2\mult_G(\alpha^2).
\end{align*}
For $\ell \in [k]$,
  \begin{align*}
|E^1(G(\ell))|+2|E^2(G(\ell))|=e_\ell+2\dg_\Theta (\ell)+\mult_{G(\ell)}(\alpha)+2\mult_{G(\ell)}(\alpha^2)=\rho_\ell.
\end{align*}
Finally, we have the following for $\ell \in [k]$, and so \eqref{colorcon2cruse} holds. 
  \begin{align*}
\dg_{G(\ell)} (\alpha) &=\mult_{G(\ell)} (\alpha, X)+ \mult_{G(\ell)} (\alpha)+ 2\mult_{G(\ell)} (\alpha^2) \\
 &=  \dg_\Theta (\ell)+d_\ell+\overline{d_\ell}+(\rho_\ell-e_\ell-2\dg_\Theta (\ell)-d_\ell-\overline{d_\ell})\\
 &=  \rho_\ell-e_\ell-\dg_\Theta (\ell)\\
& \leq  n-r.
\end{align*}

Let
\begin{align*}
    \begin{cases}
    g,f: V(\Gamma)\rightarrow \mathbb{N}\cup \{0\},\\
    g(u)=f(u)=n-r& \text{ for } u\in X,\\
    g(\ell)=\rho_\ell-e_\ell+r\dotdiv n & \text{ for } \ell \in [k], \\
    f(\ell)=\rounddown{(\rho_\ell-e_\ell-d_\ell)/2}& \text{ for }\ell \in [k].
    \end{cases}
\end{align*}
Clearly, $\Theta\subseteq \Gamma$ exists if and only if  $\Gamma$ has a $(g,f)$-factor, but by \cite[Theorem 5]{MR3564794},   $\Gamma$ has a  $(g,f)$-factor   if and only if the following  conditions hold.

\begin{align*}
\sum_{i\in I} g(i)&\leq \sum_{\ell\in [k]} \min \Big\{ f(\ell), \mult_\Gamma(\ell I)\Big \}
\quad &\forall I\subseteq X,\\
\sum_{\ell\in K}g(\ell)&\leq \sum_{i\in [r]} \min \Big\{ f(i), \mult_\Gamma(i K)\Big\}
\quad &\forall K\subseteq [k].
\end{align*}
Equivalently, we must have
\begin{align*}
(n-r)|I|&\leq \sum_{\ell\in [k]} \min \Big \{ \rounddown{ \frac{\rho_\ell-e_\ell-d_\ell}{2}}, \mu_I(\ell)\Big\}
\quad &\forall I\subseteq [r],\\
\sum_{\ell\in K}(\rho_\ell-e_\ell+r\dotdiv n)&\leq \sum_{i\in [r]} \min \Big \{ n-r, \mu_K(i)\Big \}
\quad &\forall K\subseteq [k].
\end{align*}

By   \cite[Theorem 1]{MR1081839}, $\Gamma$ has a  $(g,f)$-factor   if and only if
$$
\sum_{a\notin  A}f(a)\geq  \sum_{a\in A} \Big( g(a) \dotdiv \dg_{G-\bar A}(a) \Big) \quad \forall A\subseteq X\cup Y,
$$
or equivalently,
\begin{align*}
    \sum_{i\in \bar  I}(n-r)+\sum_{\ell\in\bar  K}\rounddown{ \frac{\rho_\ell-e_\ell-d_\ell}{2}} &\geq  \sum_{\ell\in K} \Big( \rho_\ell-e_\ell-n+r \dotdiv \mu_I(\ell) \Big) \\ & \  + \sum_{i\in I} \Big( n-r \dotdiv \mu_{K}(i) \Big) \quad \forall I\subseteq [r], K\subseteq [k].
\end{align*}

\section{Corollaries and Open Problems} 
 Recall that in order to embed an $r\times r$ symmetric $\vb*{\rho} $-latin rectangle $L$   to to  an $n \times n$ symmetric $\vb*{\rho} $-latin square  with a prescribed diagonal tail $\vb*{d}$, it is necessary that
\begin{align*}  
&e_\ell \geq \rho_\ell+d_\ell+2r-2n\quad \forall \ell \in [k], \text { and}\\
&(n-r-D-q)/2 \text { is a non-negative integer},
\end{align*}
where $D:=\sum_{\ell\in [k]} d_\ell$ and $q$ is the number of symbols $\ell\in [k]$ such that
$\rho_\ell+e_\ell+d_\ell$ is odd. 
We show that imposing slightly stronger assumptions will lead to much simpler conditions
than those of Theorem \ref{andhoffrhothm}. In our first application of Theorem \ref{andhoffrhothm}, we assume that  $e_\ell\geq \rho_\ell+r-n$ for  $\ell \in [k]$; this   in particular implies that $k\geq n+r$ for 
$$k(n-r)=\sum_{\ell\in [k]} (n-r)\geq \sum_{\ell\in [k]} (\rho_\ell-e_\ell)=(n+r)(n-r).$$ 

\begin{corollary} \label{cor1diag}
An $r\times r$ symmetric $\vb*{\rho} $-latin rectangle with $e_\ell\geq \rho_\ell+r-n$ for  $\ell \in [k]$ can be embedded in  an $n \times n$ symmetric $\vb*{\rho} $-latin square  with a prescribed diagonal tail $\vb*{d}$ if and only if  $(n-r-D-q)/2$ is a non-negative integer,  and   any of the following conditions is satisfied.
\begin{align*}
(n-r)|I|&\leq \sum_{\ell\in [k]} \min \Big \{  \rounddown{\frac{\rho_\ell-e_\ell-d_\ell}{2}}, \mu_I(\ell)\Big\}
\quad &\forall I\subseteq [r];\\
\sum_{\ell\in K} \rounddown{\frac{\rho_\ell-e_\ell-d_\ell}{2}} &\geq \sum_{i\in [r]} \Big( n-r \dotdiv \mu_{\bar K}(i) \Big) \quad &\forall K\subseteq [k].
\end{align*}
\end{corollary}
\begin{proof} Suppose that  $e_\ell\geq r-n+\rho_\ell$ for  $\ell \in [k]$. Since $d_\ell\leq n-r$ for $\ell \in [k]$,  \eqref{easyneccondiag} holds. Moreover, $\rho_\ell-e_\ell+r\dotdiv n=0$ for $\ell \in [k]$, and consequently, \eqref{longineqdial2} is satisfied.  Modifying  the proof of    \cite[Theorem 1]{MR1081839}, implies that
the graph $\Gamma[X,[k]]$ of the proof of Theorem \ref{andhoffrhothm} has a  $(g,f)$-factor  with $g(\ell)=0$ for $\ell \in [k]$  if and only if
\begin{equation*}  
    \sum_{\ell\in K}f(\ell)\geq  \sum_{i\in [r]} \Big( g(i) \dotdiv \dg_{G-K}(i) \Big) \quad \forall K\subseteq [k],
\end{equation*}
or equivalently,  
\begin{align*}
     \sum_{\ell\in K} \rounddown{\frac{\rho_\ell-e_\ell-d_\ell}{2}}\geq \sum_{i\in [r]} \Big( n-r \dotdiv \mu_{\bar K}(i) \Big)   \quad \forall K\subseteq [k].
\end{align*}
\end{proof}

In our next corollary, we assume that  $e_\ell\geq 2r+d_\ell-\rho_\ell$ for  $\ell \in [k]$; this   in particular implies that $k\leq (n^2+r^2-n+r)/(2r)$ for 
$$
r^2=\sum_{\ell\in [k]} e_\ell\geq \sum_{\ell\in [k]} (2r+d_\ell-\rho_\ell)=2kr+(n-r)-n^2.
$$
The following offers a  simpler generalization of Theorems \ref{crusorigthm} and \ref{rbyrsymprediag}; see Remarks \ref{remarkdiago1}.
\begin{corollary} \label{cor2diag}
If  $e_\ell\geq 2r+d_\ell-\rho_\ell$ for  $\ell \in [k]$, 
then an $r\times r$ symmetric $\vb*{\rho} $-latin rectangle can be embedded in  an $n \times n$ symmetric $\vb*{\rho} $-latin square  with a prescribed diagonal tail $\vb*{d}$ if and only if  $(n-r-D-q)/2$ is a non-negative integer, and  any of the following conditions is satisfied.
\begin{align*}
\sum_{\ell\in K}(\rho_\ell-e_\ell+r\dotdiv n)&\leq \sum_{i\in [r]} \min \Big \{ n-r, \mu_K(i)\Big \}
 &\forall K\subseteq [k];\\
(n-r)(r-|I|) &\geq  \sum_{\ell \in [k]} \Big( \rho_\ell-e_\ell-n+r \dotdiv \mu_I(\ell) \Big)  &\forall I\subseteq [r].
\end{align*}
\end{corollary}
\begin{proof}
 Suppose that $e_\ell\geq 2r+d_\ell-\rho_\ell$ for  $\ell \in [k]$. Since $\rho_\ell\leq n$ for $\ell \in [k]$,  \eqref{easyneccondiag} is satisfied. Moreover,     $r-e_\ell \leq (\rho_\ell-e_\ell-d_\ell)/2$ for  $\ell \in [k]$.
Thus, the following confirms that  \eqref{longineqdial1} holds. 
\begin{align*}
    (n-r)|I| &\leq (k-r)|I| =\sum_{u\in I} \dg_\Gamma (u)=\sum_{\ell\in N_\Gamma(I)} \mult_\Gamma(\ell I)\\
    &=\sum_{\ell\in N_\Gamma(I)} \min \Big\{\mult_\Gamma(\ell I), \dg_\Gamma(\ell)\Big\}\\
    &\leq   \sum_{\ell\in N_\Gamma(I)} \min \Big\{\mult_\Gamma(\ell I), \rounddown{\frac{\rho_\ell-e_\ell-d_\ell}{2}}\Big\}\\
    &=\sum_{\ell\in [k]} \min \Big \{ \rounddown{\frac{\rho_\ell-e_\ell-d_\ell}{2}} ,\mu_I(\ell)\Big\}
 &\forall I\subseteq [r].
\end{align*}

 Let 
\begin{align*}
    \begin{cases}
    g,f: V(\Gamma)\rightarrow \mathbb{N}\cup \{0\},\\
    g(u)=f(u)=n-r& \text{ for } u\in X,\\
    g(\ell)=\rho_\ell-e_\ell+r\dotdiv n & \text{ for } \ell \in [k], \\
    f(\ell)=\beta& \text{ for }\ell \in [k],
    \end{cases}
\end{align*} 
 where $\beta$ is a sufficiently large number. The graph $\Gamma$ of the proof of Theorem \ref{andhoffrhothm} has a $(g,f)$-factor if and only if 
\begin{align} \label{ineqwhatever}
(n-r)(r-|I|)+\sum_{\ell\in \bar K}\beta &\geq  \sum_{\ell\in K} \Big( \rho_\ell-e_\ell-n+r \dotdiv \mu_I(\ell) \Big) \nonumber\\ 
  &\quad + \sum_{i\in I} \Big( n-r \dotdiv \mu_{K}(i) \Big) & \forall I\subseteq X, K\subseteq [k].
\end{align}
For $K\neq [k]$,  \eqref{ineqwhatever} is trivial, and for $K=[k]$, it simplifies to the following.
\begin{align*}
(n-r)(r-|I|) &\geq  \sum_{\ell \in [k]} \Big( \rho_\ell-e_\ell-n+r \dotdiv \mu_I(\ell) \Big)  + \sum_{i\in I} \Big( n-r \dotdiv \mu_{[k]}(i) \Big) \\
&= \sum_{\ell \in [k]} \Big( \rho_\ell-e_\ell-n+r \dotdiv \mu_I(\ell) \Big)  + \sum_{i\in I} \Big( n-r \dotdiv (k-r) \Big)\\
&=\sum_{\ell \in [k]} \Big( \rho_\ell-e_\ell-n+r \dotdiv \mu_I(\ell) \Big) & \forall I\subseteq [r].
\end{align*}
\end{proof}

\begin{corollary} \label{cor3diag}
If 
\begin{align*}
    2r+d_\ell-e_\ell \leq \rho_\ell \leq n-r+e_\ell \quad \forall\ell \in [k],
\end{align*}
then an $r\times r$ symmetric $\vb*{\rho} $-latin rectangle can be embedded in  an $n \times n$ symmetric $\vb*{\rho} $-latin square  with a prescribed diagonal tail $\vb*{d}$ if and only if  $(n-r-D-q)/2$ is a non-negative integer.
\end{corollary}
\begin{proof}
Since \eqref{easyneccondiag} holds, and the inequality in  \eqref{colorcon3nodiag} is trivial, $\Theta\subseteq \Gamma$ always exists.
\end{proof}

Finally, the following result offers a  very simple generalization of Theorems \ref{crusorigthm} and \ref{rbyrsymprediag}; see Remarks \ref{remarkdiago1}.
 \begin{corollary} \label{hellcordiag}
Suppose that
\begin{align*}
    (k-r)\rounddown{\frac{\rho_\ell-e_\ell-d_\ell}{2}} &\geq (n-r)(r-e_\ell)\geq (k-r)(\rho_\ell-e_\ell-n+r) & \forall \ell \in [k];\\
    (r-e_{\ell'})\rounddown{\frac{\rho_{\ell}-e_{\ell}-d_{\ell}}{2}}&\geq (r-e_{\ell})(\rho_{\ell'}-e_{\ell'}-n+r)  & \forall \ell,\ell'\in [k].
\end{align*}
Then an $r\times r$ symmetric $\vb*{\rho} $-latin rectangle  can be completed to      an $n \times n$ symmetric $\vb*{\rho} $-latin square  with a prescribed diagonal tail $\vb*{d}$ if and only if   $(n-r-D-q)/2$ is a non-negative integer.
\end{corollary}
\begin{proof}
Using the first inequality, we have $(\rho_\ell-e_\ell-d_\ell)/2\geq \rounddown{ (\rho_\ell-e_\ell-d_\ell)/2}\geq \rho_\ell-e_\ell-n+r$
 which implies \eqref{easyneccondiag}. By \cite[Corollary 2]{MR1081839}, if for all pairs of vertices $x,y$ of the bigraph $\Gamma[X,[k]]$, 
 \begin{align} \label{hellkiineq}
    f(x)\deg_{\Gamma}(y)\geq g(y)\deg_{\Gamma}(x),
\end{align}
then $\Gamma$ has a $(g,f)$-factor (and so $\Theta\subseteq \Gamma$ satisfying \eqref{colorcon3nodiag} exists). Recall that $\dg_\Gamma(u)=k-r$, $g(u)=f(u)=n-r$ for $u\in X$, and $\dg_{\Gamma}(\ell)=r-e_\ell$, $g(\ell)=\rho_\ell-e_\ell+r\dotdiv n, f(\ell)=\rounddown{(\rho_\ell-e_\ell-d_\ell)/2}$ for $\ell \in [k]$.  Since for $x,y\in X$ \eqref{hellkiineq} is trivial, the proof is complete. 
\end{proof}

 \begin{remark} \label{remarkdiago1}\textup{
 Both Corollaries \ref{cor2diag} and  \ref{hellcordiag} generalize  the Andersen-Hoffman Theorem as well as Cruse's Theorem. To see this, let $\rho_1=\dots=\rho_k=n=k$. 
Both inequalities in Corollary \ref{hellcordiag} simplify to
$$
\rounddown{\frac{n-e_\ell-d_\ell}{2}}\geq r-e_\ell \quad \forall \ell\in[n],
$$  
which is equivalent to the hypothesis of Corollary \ref{cor2diag}, $e_\ell\geq 2r+d_\ell-n$ for $\ell \in [n]$.
Since
\begin{align*}
\sum_{\ell\in K}(\rho_\ell-e_\ell+r\dotdiv n)&= \sum_{\ell\in K}(n-e_\ell+r\dotdiv n)
= \sum_{\ell\in K}(r-e_\ell)\\
 &=\sum_{i\in [r]}  \mu_K(i)
=\sum_{i\in [r]} \min \Big \{ n-r, \mu_K(i)\Big \} \quad \forall K\subseteq [n],
\end{align*}
\begin{align*}
\sum_{\ell\in [k]} \Big( \rho_\ell-e_\ell-n+r \dotdiv \mu_{\bar I}(\ell) \Big)&=\sum_{\ell\in [k]} \Big( r-e_\ell \dotdiv \mu_{\bar I}(\ell) \Big)=\sum_{\ell\in [k]} \Big( r-e_\ell - \mu_{\bar I}(\ell) \Big)\\
&=
\sum_{\ell\in [n]} \mult_\Gamma (\ell I)=|I|(n-r) \quad \forall I\subseteq [r],
\end{align*}
the long inequalities in Corollary \ref{cor2diag} are trivial. 
Recall that $D=\sum_{\ell\in [n]} d_\ell$ and $q$ is the number of symbols $\ell\in [n]$ such that
$n+e_\ell+d_\ell$ is odd. If $D=n-r$, the  condition that $(n-r-D-q)/2=-q/2$  is a non-negative integer is equivalent to  $q=0$, and consequently, $e_\ell+d_\ell+n$ is even for $\ell\in [n]$. 
 If $D=0$, the  condition that $(n-r-D-q)/2=(n-r-q)/2$  is a non-negative integer is equivalent to  $q\leq n-r$, and consequently, the number of symbols $\ell \in [n]$ such that $e_\ell - n$ is even, is at least $r$.
}\end{remark}

A $\rho$-latin square $L$ is {\it diagonal} if each $\ell \in [k]$ occurs at most once on the diagonal, and is {\it idempotent} if $L_{ii}=i$ for $i\in [n]$. Completing partial idempotent latin squares has a rich history (see \cite {MR2961221}). Theorem \ref{andhoffrhothm} in particular settles the necessary and sufficient conditions that ensure an $r\times r$ symmetric diagonal (or idempotent) $\vb*{\rho} $-latin rectangle $L$ can be embedded in  an $n \times n$ symmetric diagonal (or idempotent) $\vb*{\rho} $-latin square, but the following problem remains  unsolved.
\begin{question} \label{rhoidempprob}
Find necessary and sufficient conditions that ensure that an idempotent $r \times s$ $\vb*{\rho} $-latin rectangle can be extended to an idempotent  $n\times n$ $\vb*{\rho}$-latin square.
\end{question}
We remark that Problem \ref{rhoidempprob} is open even for latin squares \cite{MR2961221}. The most general result up to date   is Rodger's theorem that  settles the problem for latin squares when $r=s$, $n\geq 2r$  \cite{MR0731592,MR0755043}. Two other notable results are  \cite{MR0667240} and \cite{MR4477846}.

Let $\vb*{\rho} =(\rho_1,\dots,\rho_k)$ with $1\leq \rho_\ell\leq n\leq k$ for $\ell \in [k]$ such that $\sum_{\ell \in [k]} \rho_\ell=n^2$. A {\it partial $\vb*{\rho} $-latin square} $L$ of  order $n$ is an $n\times n$ array that is partially filled using $k$ different symbols each occurring at most once in each row and at most once in each column, such that each symbol $i$ occurs at most $\rho_\ell$ times in $L$ for $\ell \in [k]$. We say that $L$ is {\it critical} if it can be extended to exactly one  $\vb*{\rho} $-latin square of order $n$, but removal of any entry of $L$ destroys the uniqueness of the extension, and the number of non-empty cells of $L$ is the {\it size} of the critical partial  $\vb*{\rho} $-latin square. 
\begin{question}
Find  good bounds  for the smallest and largest sizes of critical partial $\vb*{\rho} $-latin squares. 
\end{question}
For latin squares, it is known that the size of the largest critical set is between and $n^2-O(n^{5/3})$ and $n^2/-7n/2+o(n)$
, and it is conjectured that the size of the smallest critical set is $\rounddown{n^2/4}$; see 
\cite[Section 1.8]{MR2246267}.



\section*{Acknowledgements}

We wish to thank the anonymous referee for  very carefully reading this manuscript and for many suggestions.


%
%

\bibliographystyle{alphaurl}

\end{document}